\documentclass[11pt,fleqn]{article} 

\usepackage[affil-sl]{authblk}
\usepackage[breaklinks=true]{hyperref}
\usepackage{breakcites}

\usepackage{titlesec}
\titleformat{\section}
  {\large\center\bfseries\scshape}
  {\thesection.}{.7em}{}
\titlespacing*{\section}{0pt}{3.5ex plus 0ex minus 0ex}{1.5ex plus 0ex}

\usepackage[english]{babel}
\usepackage{amsfonts}
\usepackage{mathrsfs}
\usepackage{bbm}
\usepackage{latexsym}
\usepackage{amssymb}
\usepackage{amsmath}

\usepackage{todonotes}
\usepackage{enumitem}
\setlist{nolistsep}
\usepackage{amsthm}
\usepackage[capitalize]{cleveref}

\newtheoremstyle{plain}
{2mm}	
{2mm}	
{\slshape}	
{}	
{\color{black}\bfseries}	
{.}	
{.5em}	
{}	
\newtheoremstyle{definition}
{2mm}
{2mm}
{}
{}
{\color{black}\bfseries}
{.}
{.5em}
{}
\theoremstyle{plain}
	
\newtheorem{Theorem}{Theorem}[section]
\newtheorem{Lemma}[Theorem]{Lemma}
\newtheorem{Proposition}[Theorem]{Proposition}

\theoremstyle{definition}
\newtheorem{Definition}[Theorem]{Definition}
\newtheorem{Remark}[Theorem]{Remark}
\newtheorem{Example}[Theorem]{Example}

\theoremstyle{plain} 
\newcounter{MainTheoremCounter}

\newtheorem{Maintheorem}[MainTheoremCounter]{Theorem}

\theoremstyle{plain}
\newtheorem*{namedthm}{\namedthmname}
\newcounter{namedthm}
\makeatletter
	\newenvironment{named}[2]
	{\def\namedthmname{#1}
	\refstepcounter{namedthm}
	\namedthm[#2]\def\@currentlabel{#1}}
	{\endnamedthm}
\makeatother

\numberwithin{equation}{section}

\usepackage{xcolor}

\definecolor{Scarlet}{rgb}{0.78, 0.11, 0.0}
\definecolor{Blue}{rgb}{0.0, 0.42, 0.47}

\hypersetup{	citecolor = Scarlet,colorlinks,
			linkcolor = black,
			urlcolor = Scarlet}

\newcommand{\N}{\mathbb{N}}
\newcommand{\Z}{\mathbb{Z}}

\newcommand{\Define}[1]{\textbf{{#1}}}

\renewcommand{\epsilon}{\varepsilon}
\renewcommand{\leq}{\leqslant}
\renewcommand{\geq}{\geqslant}
\renewcommand{\setminus}{\backslash}

\newcommand{\F}{\mathcal{F}}
\newcommand{\FilterClosure}[1]{\overline{#1}}

\newcommand{\G}{\mathbf{G}}

\newcommand{\allNonemptyFiniteSubsetsOf}[1]{\mathcal{P}_f (#1)}
\newcommand{\FS}{\allNonemptyFiniteSubsetsOf{\N}}

\newcommand{\A}{\mathscr{A}}
\newcommand{\W}{\mathcal{W}}
\newcommand{\low}{\text{low}}
\newcommand{\new}{\text{new}}
\renewcommand{\P}{\mathbf{P}}
\newcommand{\Dshift}[1]{\Delta_{#1}}
\newcommand{\B}{\mathcal{B}}
\newcommand{\FF}{\mathbb{E}}

\begin{document}

\title{Revisiting the nilpotent polynomial Hales-Jewett theorem}
\author{John H.~Johnson~Jr.~and~Florian Karl Richter\\
  {%
    \small
    \href{mailto:johnson.5316@osu.edu}{\url{johnson.5316@osu.edu}} and %
    \href{mailto:richter.109@osu.edu}{\url{richter.109@osu.edu}}
  }}
\affil{\small Department of Mathematics\\
  The Ohio State University\\
  Columbus, Ohio
  }
\date{}
\maketitle
\begin{abstract}
Answering a question posed by Bergelson and Leibman
in \cite{BL03}, we establish a nilpotent version of the
polynomial Hales-Jewett theorem that contains the main theorem in \cite{BL03} as a special case.
Important to the formulation and the proof of our main theorem is
the notion of a relative syndetic set (relative with respect to a closed non-empty subsets of $\beta\G$) \cite{SZZ09}. 
As a corollary of our main theorem we prove an extension
of the restricted van der Waerden Theorem to nilpotent groups,
which involves nilprogressions.
\end{abstract}


\tableofcontents

\section{Introduction}\label{section:intro}

Van der Waerden's Theorem on arithmetic progressions \cite{vdW28} is one of the oldest and most well known results in Ramsey theory. One equivalent formulation, which our result is formally similar to, is due to Kakeya and Morimoto \cite[Theorem~I]{Kakeya:1930ts} and involves the notion of a syndetic set:

\begin{Definition}
\label{def:syndetic-sets}
Let $(\G,\cdot)$ be a group. A set $A\subset\G$ is called \Define{syndetic} if there exits a finite non-empty set $K\subset \G$ such that $K^{-1} A=\G$, where $K^{-1}A=\{k^{-1}a: k\in K, a\in A\}$.
\end{Definition}

\begin{named}{Van der Waerden's Theorem}{\cite{vdW28,Kakeya:1930ts}}
\label{theorem:vdW}
Every syndetic subset of the integers contains arbitrarily long arithmetic progressions.
\end{named}

Using a dynamical approach, Furstenberg and Weiss \cite{FW78} extended van der Waerden's theorem to arbitrary abelian groups
and restricted the arithmetic structure to IP-sets. 
In the following definition and in the rest of this paper we use $\allNonemptyFiniteSubsetsOf{X}$ to denote the collection of all non-empty finite subsets of a set $X$. For $\alpha,\beta\in\FS$ we write $\alpha<\beta$ if $\max \alpha < \min\beta$.

\begin{Definition}
\label{def:IP-sets}
Let $(\G,\cdot)$ be a group. A map $x:\FS\to \G$ is called an $\Define{IP mapping}$ if for all $\alpha,\beta\in\FS$ with $\alpha<\beta$ one has $x(\alpha\cup\beta)=x(\alpha)\cdot x(\beta)$.
\end{Definition}

\begin{Theorem}[IP van der Waerden Theorem, {cf. {\cite[Section 3]{FW78}},~{\cite[Subsection 2.5]{Furstenberg81}} and {\cite[Subsection 1.5]{McCutcheon99}}}]
\label{theorem:IPvdW}
Let $k\in\N$, let $(\G,+)$ be an abelian group and let
$x_1,\ldots,x_k:\FS\to\G$ be IP mappings.
Then for any syndetic set $A\subset \G$
there are $\alpha\in\FS$ and $a\in\G$ such that
$\{a+x_1(\alpha),\ldots,a+x_k(\alpha)\}\subset A$.
\end{Theorem}


It is natural to ask if there are extensions of \cref{theorem:IPvdW} to non-abelian groups. We note that there is a version of \cref{theorem:IPvdW} for arbitrary groups (actually arbitrary semigroups, \cite{JohnsonJr:2015jp}), but if the underlying group structure is non-commutative then the classical arithmetic arrangement $\{x_1(\alpha)a,\ldots,x_k(\alpha)a\}$ is not generally guaranteed. 

However, in the case of nilpotent groups it is.
By interpreting IP mappings as ``polynomial mappings of degree 1'', Bergelson and Leibman in \cite{BL03} used this insight to prove a powerful polynomial extension of \cref{theorem:IPvdW} for nilpotent groups. To state their result, we list a few more definitions:

\begin{Definition}\label{def:pm}
\
\begin{enumerate}
[label=(\alph{enumi}),ref=(\alph{enumi}),leftmargin=*]
\item
We define
$
\FF:=\big\{V\subset \FS: \exists \beta~\text{s.t.}~ \{\alpha\in\FS: \alpha>\beta\}\subset V\}.
$
\item
Let $(\G,\cdot)$ be a group.
\begin{itemize}
\item[(i)]
Let $x,y:\FS\to\G$ be two mappings. We say $x=y$ $\FF$-a.e. if and only if $\{\alpha\in\FS:x(\alpha)=y(\alpha)\}\in\FF$.
\item[(ii)]
Let $x:\FS\to\G$ and $\beta\in\FS$. The \Define{(discrete) derivative (in direction $\beta$)} is the map 
$D_\beta x:\FS\to\G$ defined by
\[
D_\beta x(\alpha)=(x(\alpha))^{-1}x(\alpha\cup\beta)(x(\beta))^{-1}.
\]
\item[(iii)] Let $P:\FS\to\G$. We call $P$ a \Define{polynomial mapping of degree $1$} if there exists an IP map $x:\FS\to\G$ such that $P=x$ $\FF$-a.e.

Recursively, for $d\in\N$ with $d>1$, we call $P$ a  \Define{polynomial mapping of degree $d$} if and only if
$$
\{\beta\in\FS: D_\beta P~\text{is a polynomial mapping of degree $d-1$}\}\in\FF.
$$
\end{itemize} 
\end{enumerate} 
\end{Definition}

\begin{Theorem}[Polynomial IP van der Waerden Theorem for Nilpotent Groups, {\cite[Theorem 4.4]{BL03}}]\label{theorem:Poly-IPvdW-nilpotent}
Let $(\G,\cdot)$ be a nilpotent group, let
$P_1,\ldots,P_k:\FS\to\G$ be polynomial mappings and let 
$A\subset \G$ be syndetic. Then there are $a\in\G$ and
$\alpha\in\FS$ such that
$\{ P_1(\alpha) a,\ldots, P_k(\alpha) a\}\subset A$.
\end{Theorem}

Another fundamental result in Ramsey Theory is Bergelson's and Leibman's polynomial Hales-Jewett Theorem (PHJ) \cite{BL99}.
One can view \cref{theorem:Poly-IPvdW-nilpotent} as a partial extension of PHJ to nilpotent groups, but \cref{theorem:Poly-IPvdW-nilpotent} doesn't contain PHJ as a special case.
In \cite[Remark 6.4]{BL03} it was asked by Bergelson and Leibman if it is possible to formulate and prove a
``full-fledged nilpotent Polynomial Hales-Jewett Theorem''.



In this paper we offer an affirmative answer to this question in the
form of \cref{theorem:3} below.
We use the notion of ``filter-syndetic sets'' (introduced by Shuungula, Zelenyuk and Zelenyuk \cite{SZZ09}) and the notion of idempotent filters.


\begin{Definition}
Let $(\G,\cdot)$ be a group.
\begin{enumerate}
[label=(\alph{enumi}),ref=(\alph{enumi}),leftmargin=*]
\item
If $\F$ is a collection of subsets of a set $X$, then $\F$ is called a
\Define{filter} on $X$ if it satisfies
\begin{itemize}
 \item $\emptyset\notin\F$ and $X\in\F$;
 \item if $A\in\F$ and $B\supset A$ then $B\in\F$;
 \item if $A,B\in\F$ then $A\cap B\in\F$.
\end{itemize}
\item
Given two filters $\F$ and $\mathcal{G}$ on $\G$ we define the \Define{filter product} $\F\cdot \mathcal{G}$
according to the rule
\begin{equation}
\label{equation:filterproduct}
A\in\F\cdot\mathcal{G} \quad\Leftrightarrow\quad
\{x \in \G:\{y \in \G:x\cdot y\in A\}\in\mathcal{G}\}\in\F.
\end{equation}
It can easily be checked that the filter product of two filters
is itself a filter.
A filter $\F$ is called \Define{idempotent} if it
satisfies $\F\cdot\F\supset\F$.
A special class of idempotent filters is the class
of idempotent ultrafilers\footnote{An
\Define{ultrafilter} is a maximal filter, i.e., a filter that
is not properly contained in another filter.}.
Idempotent filters and idempotent ultrafilters have been studied extensively due to their applicability to Ramsey Theory and Ergodic Ramsey Theory (see \cite{Bergelson96,Bergelson10,BH89,BM00, BM10, FK85, HS12, Tserunyan14}).
\item
A set $A\subset\G$ is called  \Define{$\F$-syndetic} if for every $V\in\F$ there exists $K\in \allNonemptyFiniteSubsetsOf{V}$ such that $K^{-1}A\in\F$ (cf. \cite[Section 2]{SZZ09}).
Note that regular syndeticity (as introduced in \cref{def:syndetic-sets}) corresponds to the special case where $\F$ equals the trivial filter on $\G$, i.e. $\F=\{\G\}$.
\end{enumerate}
\end{Definition} 

For \cref{theorem:3} we need to restrict the class of polynomial mappings to those which are detectable by a given filter $\F$. This leads to the following definition.

\begin{Definition}\label{def:gcofmp}
Let $(\G,\cdot)$ be a group and let $\F$ be a filter on $\G$. For $u,c\in\G$ we write $u^c$ for the conjugate $c^{-1}uc$ and $[u,c]$ for the commutator $u^{-1}c^{-1}uc$.
\begin{enumerate}
[label=(\alph{enumi}),ref=(\alph{enumi}),leftmargin=*]
\item
Let $1_\G$ denote the identity element of $\G$. We say a map $x:\FS\to\G$ is \Define{$\F$-measurable} (or \Define{$(\FF,\F)$-measurable})
if for all $V\in\F$ the set $x^{-1}(V\cup\{1_\G\})$ belongs to $\FF$.
\item
Let $\F$ be a filter on $\G$ and let $\P$ be a collection of polynomial mappings taking values in the group $\G$. We say that $\P$ is a \Define{good collection of $\F$-measurable polynomial mappings} if
\begin{enumerate}
[label=(\Roman{enumii}),ref=(\Roman{enumii}),leftmargin=*]
\item\label{item:poly0}
$1_\G\in\P$ (by abuse of language we use $1_\G$ to denote both the identity in $\G$ and the constant polynomial mapping $\alpha\mapsto 1_\G$);
\item\label{item:poly1}
every $P\in\P$ is $\F$-measurable;
\item\label{item:poly2}
for all $R,P\in\P$, $\{c\in\G: R^c [c,P]\in\P\}\in\F$;
\item\label{item:poly3}
for all $R,P\in\P$ there exists $C\in\F$ and $\mathcal{T}\in\FF$ such that 
$(R D_\beta P)^c\in\P$ for all $c\in C$ and $\beta\in\mathcal{T}$.
\end{enumerate}
\end{enumerate}
\end{Definition}

We remark that conditions \ref{item:poly0} and \ref{item:poly1} above are  natural in our setting whereas conditions \ref{item:poly2} and \ref{item:poly3} are technical necessities needed to perform a localized color focusing argument in
\cref{lemma:color-focusing} below.
We refer the reader to
\cref{section:nilpotent-PHJ} for a concrete example of a good collections of $\F$-measurable polynomial mappings on a nilpotent group.

\begin{Maintheorem}
[Nilpotent PHJ]
\label{theorem:3}
Let $\F$ be an idempotent filter on a nilpotent
group $(\G,\cdot)$ and let $\P$ be a
good collection of $\F$-measurable polynomial mappings.
Then for all $P_1,\ldots,P_k\in\P$ and all
$\F$-syndetic sets $A$ there are $\alpha\in\FS$ and $a\in\G$
such that $\{P_1(\alpha)a,\ldots,P_k(\alpha)a\}\subset A$.
\end{Maintheorem}

By choosing $\F$ to be the trivial filter $\{\G\}$, we see that \cref{theorem:3} contains \cref{theorem:Poly-IPvdW-nilpotent} as a special case. We also claim that \ref{theorem:PHJ} (which we recall in \cref{section:HJ}) can be derived quickly from \cref{theorem:3}. The details are provided in \cref{section:HJ}.


As a new combinatorial application of \cref{theorem:3} we formulate and prove an extension of the so-called ``Restricted van der Waerden Theorem'' \cite{Spencer75,Nesetril:1976tl} to nilpotent groups (see \cref{section:rexvdW}).

\textbf{Acknowledgements:} We thank Vitaly Bergelson for reading an
early draft of this paper and providing several helpful comments and
additional references, and also Joel Moreira for several helpful
discussions. We thank the referee for helpful suggestions on improving and streamlining the exposition of this paper.

\section{Topological algebra of closed subsemigroups of $\beta\G$}
\label{section:top-alg}

For a group $(\G,\cdot)$, let
$\beta\G$ denote the collection of all ultrafilters on $\G$.
Note that $\beta\G$ is a semigroup when endowed with the operation given by equation \eqref{equation:filterproduct}.
Given a subset $A\subset \G$ we define $\overline{A}:=\{p\in \beta\G: A\in p\}$. It is well known that
$\beta\G$ endowed with the topology generated by
$\{\overline{A}: A\subset \G\}$ is a compact Hausdorff
right topological semigroup (see
\cite{HS12} for a comprehensive discussion on the
topological and algebraical aspects of $\beta\G$).

There exists a natural one-to-one correspondence 
between filters on $\G$ and non-empty closed subsets of $\beta\G$:
If $T\subset \beta\G$ is non-empty and closed then the filter
associated with $T$ is defined as
$\F_T:=\{A\subset\G: T\subset \overline{A} \}$.
Vice versa, if $\F$ is a filter on $\G$ then
the \Define{closure of the filter} $\F$, defined as the set
$\FilterClosure{\F} := \bigcap_{A\in \F} \overline{A}$,
is a non-empty and closed subset of $\beta\G$. 
Note that $\FilterClosure{\F}$ is the collection of all ultrafilters that extend $\F$.
Clearly, the closure of the filter $\F_T$ is $T$
and the filter associated with $\FilterClosure{\F}$ is $\F$.

We are particularly interested in idempotent filters.
If $\F$ is idempotent then the closure
$T=\FilterClosure{\F}$ is a
closed subsemigroup of $\beta\G$. Note that
the reverse is not true; there are closed
subsemigroups whose corresponding filter is
not idempotent (see \cite{Davenport90}
for a complete combinatorial characterization
of closed subsemigroups of $\beta\G$).

Any compact Hausdorff right topological semigroup $T$ has a smallest two sided ideal $K(T)$, which is the union of all minimal left ideals and is also the union of all minimal right ideals (see \cite{HS12}). Also, by the Ellis-Numakura theorem (\cite{Ellis58,Numakura52}), every compact Hausdorff right topological semigroup $T$ contains at least one idempotent element. We denote the collection of all idempotent elements in $T$ by $E(T)$.



\begin{Definition}\label{theorem:SZZ-2.3}
Let $T=\FilterClosure{\F}$ be a closed subsemigroup of $\beta\G$ and let
$A\subset\G$. The set $A$ is called \Define{piecewise $\F$-syndetic} if $\overline{A} \cap K(T) \ne \emptyset$.  (See \cite[Section 2]{SZZ09} for a combinatorial characterization of piecewise $\F$-syndetic sets.)
\end{Definition}

\begin{Remark}
\label{rem:partition-regularity-pw-f-syndetic}
It is a consequence of \cref{theorem:SZZ-2.3} that
if $A$ is piecewise $\F$-syndetic and $A$ is partitioned into finitely
many classes then at least one of the classes is
piecewise $\F$-syndetic.
\end{Remark}

Before we end this section, let us formulate an
algebraic connection between $\F$-syndetic sets and the smallest ideal of a closed subsemigroup.

\begin{Theorem}[Theorem 2.2, \cite{SZZ09}]\label{theorem:SZZ-2.2}
Let $T=\FilterClosure{\F}$ be a closed subsemigroup of $\beta\G$ and let
$p\in K(T)$. Then for any set $A\in p$ the set
$A/p=\{x: x^{-1}A\in p\}$ is $\F$-syndetic.
\end{Theorem}


\begin{Remark}\label{rem:equivalent-statement-theorem-2}
Using the finite intersection property for ultrafilters, it is easy to show that if $A/p$ contains an arrangement of the form $\{P_1(\alpha)a,\ldots,P_k(\alpha)a\}$ for some $a\in\G$ and some $\alpha\in\FS$, then $A$ contains an arrangement of the form $\{P_1(\alpha)b,\ldots,P_k(\alpha)b\}$ for some $b\in \G$.
In view of \cref{theorem:SZZ-2.2} it is therefore clear that in \cref{theorem:3} one can replace `$\F$-syndetic' with `piecewise $\F$-syndetic'.
\end{Remark}

\section{Connections to the polynomial Hales-Jewett Theorem}
\label{section:HJ}

In this section we prove that \cref{theorem:3} implies \ref{theorem:PHJ}.

\ref{theorem:PHJ} has many equivalent forms
(cf. \cite{BL99,BL03,McCutcheon99,Walters00}), one of which is the following.

\begin{named}{PHJ}{{\cite[Theorem PHJ]{BL99}}}
\label{theorem:PHJ}
Let $r,d,k\in\N$ and let $V:=\N^d\times\{1,\ldots,k\}$.
For any $r$-coloring of $\allNonemptyFiniteSubsetsOf{V}$ there exists $b\in \allNonemptyFiniteSubsetsOf{V}$ and $\alpha\in\FS$ such that $b\cap \big(\alpha^d\times\{1,\ldots,k\}\big)=\emptyset$ and the sets
$$
b,~b\cup \big(\alpha^d\times\{1\}\big),~b\cup \big(\alpha^d\times\{2\}\big),\ldots,~ b\cup\big(\alpha^d\times\{k\}\big)
$$
are all of the same color.
\end{named}

\begin{Proposition}\label{proposition:monomial->PHJ}
\cref{theorem:3} implies \ref{theorem:PHJ}.
\end{Proposition}

\begin{proof}
Let $r,d,k\in \N$ be arbitrary.
Let $(\G,+)$ be the free abelian group in
$\N^d\times \{1,\ldots,k\}$ generators, which we denote by $(e_{n,i})_{n\in\N^d,i\in\{1,\ldots,k\}}$.
For $\gamma\in\allNonemptyFiniteSubsetsOf{\N^d}$ define
$
e_{\gamma,i}:=\sum_{n\in\gamma}e_{n,i}.
$
Let
$$
U_N:=\left\{\sum_{i=1}^k e_{\gamma_i,i}:
\gamma_i \in \allNonemptyFiniteSubsetsOf{\N^d\setminus \{1,\ldots,N\}^d},~
\gamma_i\cap\gamma_j=\emptyset~\text{for}~i\neq j \right\}
$$
and let $\F$ be defined as
\begin{equation*}
\F:=\left\{A\subset \G: \exists N~\text{s.t.}~
U_N\subset A\right\}.
\end{equation*}
It is straightforward to check that
$\F$ is an idempotent filter on $\G$.

For $\beta_1,\ldots,\beta_d\in\FS$ define $M(\beta_1,\ldots,\beta_d):=\beta_1\times\ldots\times\beta_{d}$.
We refer to maps of the form $\alpha\mapsto M(\alpha,\beta_2,\ldots,\beta_d)$, $\alpha\mapsto M(\beta_1,\alpha,\beta_3,\ldots,\beta_d)$, $\ldots$, $\alpha\mapsto M(\beta_1,\ldots,\beta_{d-1},\alpha)$
as set-monomials of degree $1$.
We refer to maps of the form $\alpha\mapsto M(\alpha,\alpha,\beta_3,\ldots,\beta_d)$, $\alpha\mapsto M(\alpha,\beta_2,\alpha,\beta_4,\ldots,\beta_d)$, $\ldots$, $\alpha\mapsto M(\beta_1,\ldots,\beta_{d-2},\alpha,\alpha)$ as set-monomials of degree $2$. Similarly we define set-monomials of degree $3,\ldots,d$. (We view the map $\alpha\mapsto\emptyset$ as the set-monomial of degree $0$.)
We say that two set-monomials $M_1$ and $M_2$ are disjoint if there exists $\alpha_0\in\FS$ such that $M_1(\alpha)\cap M_2(\alpha)=\emptyset$ for all $\alpha>\alpha_0$.

Let $\P$ denote the collection of all maps of the form $\alpha\mapsto e_{M_1(\alpha),1}+\ldots+e_{M_k(\alpha),k}$ where $M_1,\ldots,M_k$ are pairwise disjoint set-monomials of degree $\leq d$.
Clearly, all maps in $\P$ are $\F$-measurable polynomial mappings. Hence $\P$ satisfies part \ref{item:poly0} and \ref{item:poly1} of \cref{def:gcofmp}. Note that part \ref{item:poly2} of \cref{def:gcofmp} is only meaningful for non-abelian groups, as any collection of polynomial mappings taking values in an abelian group trivially satisfies this condition. Finally, we leave it to the reader to verify that $\P$ satisfies part \ref{item:poly3} of \cref{def:gcofmp}. Hence, $\P$ is a good collection of $\F$-measurable polynomial mappings.

Let $\phi:U_1\to V$ denote the map defined by
$$
\phi\left(\sum_{i=1}^k e_{\gamma_i,i}\right):=\big(\gamma_1\times\{1\}\big)\cup\big( \gamma_2\times\{2\}\big)\cup\ldots\cup\big(\gamma_k\times\{k\}\big).
$$
Now suppose we are given an arbitrary $r$-coloring of $V$. For convenience we view this finite coloring as a finite partition $V=\bigcup_{j=1}^r C_j$. For $j=1,\ldots,r$ define $D_j:=\phi^{-1}(C_j)$ and observe that $U_1=\bigcup_{j=1}^r D_j$. Since $\F$ is idempotent, any set that is contained in $\F$ is automatically piecewise $\F$-syndetic. In particular, $U_1$ is piecewise $\F$-syndetic because $U_1\in\F$. 
It thus follows from \cref{rem:partition-regularity-pw-f-syndetic} that there exists $j_0\in\{1,\ldots,r\}$ such that $D_{j_0}$ is also piecewise $\F$-syndetic.

Consider the polynomial mappings $P_0,P_1,\ldots,P_k\in\P$ where $P_0(\alpha):=1_\G$ and
$$
P_i(\alpha):=e_{\alpha^d,i},\qquad \text{for}~i\in\{1,\ldots,k\}.
$$
In light of \cref{theorem:3} and \cref{rem:equivalent-statement-theorem-2} we can find $\alpha\in\FS$ and $a\in \G$ such that
$$
a+P_0(\alpha),~a+P_1(\alpha),~a+P_2(\alpha),\ldots,a+P_k(\alpha)\in D_{j_0}.
$$
Take $b:=\phi(a)$.
It follows from $a\in U_1$ and $a+P_i(\alpha)\in U_1$ that $b\cap \big(\alpha^d\times\{1,\ldots,k\}\big)=\emptyset$. Also, $\phi\big(a+P_0(\alpha)\big)=\phi(a)=b$ and $\phi\big(a+P_i(\alpha)\big)=b\cup\big(\alpha^d\times\{i\}\big)$ for all $i\in\{1,\ldots,k\}$, which proves that
$$
b,~b\cup \big(\alpha^d\times\{1\}\big),~b\cup \big(\alpha^d\times\{2\}\big),\ldots,~ b\cup\big(\alpha^d\times\{k\}\big)\in C_{j_0}.
$$
\end{proof}

\section{Extending the restricted van der Waerden theorem to nilpotent groups}
\label{section:rexvdW}

It was shown by Spencer \cite{Spencer75} that there exists a set $V\subset \N$ containing no $k+1$ term arithmetic progressions and such that for any partition of $V$ into finitely many classes, some class must contain a $k$-term arithmetic progression.
This result is known as the restricted van der Waerden Theorem\footnote{A similar result was independently obtained by Ne{\v s}et{\v r}il and R{\"o}del \cite{Nesetril:1976tl}.}.

Using \cref{theorem:3} we can extend this result to nilpotent groups. In this extension
the role of arithmetic progressions is taken over by so-called
nilprogressions. Nilprogressions are a well studied object
that emerged from various generalizations of Freiman's theorem
to non-abelian groups \cite{BG11,BGT12,BGT13,BT16}.
For their definition let $\Sigma_{< k}$ denote the collection of
all words $w(*_1,\ldots,*_d)$ in the letters $*_1,\ldots,*_d$
such that every letter
$*_i$ appears at most $(k-1)$ times.
Also, given a word $w(*_1,\ldots,*_d)$ and
elements $x_1,\ldots,x_d$ in a group $(\G,\cdot)$
we use $w(x_1,\ldots,x_d)$ to denote the group element of
$\G$ obtained by replacing all occurrences
of the variable $*_i$ in the word $w(*_1,\ldots,*_d)$
with $x_i$.
Define a \Define{nilprogression of step $s$, length $k$ and rank $d$}
to be a set of the from
$$
A:=\{w(x_1,\ldots,x_d)a:w\in\Sigma_{< k+1}\}
$$
where $a,x_1,\ldots,x_d$ are elements in an $s$-step nilpotent group
$G$. If $|A|=|\Sigma_{< k+1}|$ then we call $A$
a \Define{non-degenerated} nilprogression.

\begin{Maintheorem}[Restricted van der Waerden Theorem
for nilprogressions]
\label{theorem:rest-nil-vdW}
For every $k\geq 1$ there exists a
$k$-step nilpotent group $(\G,\cdot)$ in two generators
and a set $V\subset\G$ with the property that $V$
does not contain any non-degenerated
nilprogressions of step $k$, length $k+1$
and rank $2$ but for any partition
of $V$ into finitely many classes, some class contains a
non-degenerated nilprogressions of step $k$,
length $k$ and rank $2$.
\end{Maintheorem}
We conjecture that analogues of \cref{theorem:rest-nil-vdW} for
nilprogressions of rank $d>2$ also hold and can be derived from
\cref{theorem:3}, however we don't attempt to prove this conjecture in this paper. (Extending our current proof to prove this generalization seems to require constructing a $k$-step nilpotent group in $d$ generators $x_1,\ldots,x_d$ where one can explicitly calculate all words $w(x_1,\ldots,x_d)\in \Sigma_{< k+1}$.)

\begin{proof}[Proof of \cref{theorem:rest-nil-vdW}]
Fix $k\geq 1$. Let $\Z[x]$ denote the collection of all polynomials
with integer coefficients, let $S:\Z[x]\to \Z[x]$ denote
the map $S(p(x))=p(x+1)$ and let $R:\Z[x]\to\Z[x]$ denote the
map $R(p(x))=p(x)+x^k$. Let $\G$ denote the group
generated by $S$ and $R$. It is well known and straight-forward
to check that $\G$ is a $k$-step nilpotent group.

Let $\Sigma_{< k+1}$
denote the collection of all words in the letters
$*_1$ and $*_2$ in which each variable $*_1$ and $*_2$
occurs at most $k$ times.
We now claim that for all $w_1,w_2\in\Sigma_{< k+1}$ if $w_1(R,S)= w_2(R,S)$ then $w_1(*_1,*_2)= w_2(*_1,*_2)$.

To prove this claim we start with a remark. For any $w\in\Sigma_{< k+1}$ one can write
$$
w(R,S)=
S^{v_{0}}R^{u_{1}}S^{v_{1}}R^{u_{2}}S^{v_{2}}\cdots
R^{u_{\ell}}S^{v_{\ell}}R^{u_{\ell+1}}
$$
with the conditions
\begin{itemize}
\item
$u_{\ell+1},v_{0}\in\{0,1,\ldots,k\}$;
\item
$\sum_{j=0}^\ell v_{j}\leq k$;
\item
$\sum_{j=1}^{\ell+1} u_{j}\leq k$.
\end{itemize}
Put $m_{j}=v_{0}+v_{1}+\ldots+v_{j-1}$.
Then the polynomial $x^k$ evaluated by the map $w(R,S)$ yields
$$
w(R,S)x^k=
u_{1} (x+m_{1})^k +\ldots+
u_{\ell+1}(x+m_{\ell+1})^k.
$$
Observe that $\{(x+m)^k:m\in \{0,1,\ldots,k\}\}$
forms a linearly independent subset of $\Z[x]$.

Now, let $w_1,w_2 \in \Sigma_{< k+1}$ with $w_1(R,S)=w_2(R,S)$.
From
$$
w_1(R,S)x^k=
u_{1,1} (x+m_{1,1})^k +\ldots+
u_{1,\ell+1}(x+m_{1,\ell+1})^k
$$
and
$$
w_2(R,S)x^k=
u_{2,1} (x+m_{2,1})^k +\ldots+
u_{2,\ell+1}(x+m_{2,\ell+1})^k
$$
it follows that $w_1(R,S)=w_2(R,S)$ if and only if $u_{1,j}=u_{2,j}$
and $m_{1,j}=m_{2,j}$ for all $j$. However, from $m_{1,j}=m_{2,j}$
it follows that $v_{1,j}=v_{2,j}$ and therefore
$w_1(*_1,*_2)=w_2(*_1,*_2)$.
This finishes the proof of the claim.

This shows that $\G$ admits non-degenerated
nilprogressions of length $k$ and rank $2$.
On the other hand, it is clear that $\G$ does not admit
non-degenerated nilprogressions of length $(k+1)$ and rank $2$, because the family
$\{(x+m)^k:m\in \{0,1,\ldots,k+1\}\}$
does not form a linearly independent subset of $\Z[x]$.

Next, let $(\G_n)_{n\in\N}$ be $\N$-many identical copies of
$\G$, let $R_n$ and $S_n$ denote identical copies of the
maps $R$ and $S$ and suppose $\G_n$ is generated by $R_n$ and $S_n$.
Let $\G_\infty:=\bigoplus_{n\in\N}\G_n$.
For convenience we identify $\G_n$ with its
embedding into $\G_\infty$, which allows us to view $R_n$ and $S_n$
as elements in $\G_\infty$.

For $\alpha\in\FS$ define
$$
S_\alpha:=\prod_{n\in\alpha}S_n\qquad\text{and}\qquad
R_{\alpha}:=\prod_{n\in\alpha}R_n.
$$
For $\gamma\in\FS$ let
$$
U_\gamma:=\left\{S_{\alpha}R_{\beta}:\alpha,\beta>\gamma \right\}
$$
and let
$$
\F:=\{U\subset \G_\infty:\exists\gamma~\text{such that}~
U_{\gamma}[\G_\infty,\G_\infty]\subset U\}.
$$
Note that $\F$ is an idempotent filter. This is easy to
see if one interprets $\F$ as the pull-back
of an idempotent filter on the abelian group
$\G_\infty/[\G_\infty,\G_\infty]$ under the natural quotient map.

Let $V:=U_{\{1\}}$.
We make the following claim, which will finish the proof of
\cref{theorem:rest-nil-vdW}: 
For any partition of $V$ into finitely many classes, some class contains a non-degenerated nilprogression of length $k$ and rank $2$.

Define
$$
\P:=\{\alpha\mapsto w(R_\alpha,S_\alpha): w\in\Sigma_{< k+1} \}.
$$
Then $\P$ is a good collection of $\F$-measurable polynomial mappings, which can be shown by routine (but somewhat lengthly) calculation.
Therefore, using \cref{theorem:3}, for every
partition of $V$ into finitely many classes we can find $a\in\G_\infty$
and $\alpha\in\FS$ such that 
$$
\{w(R_\alpha,S_\alpha) a : w\in\Sigma_{< k+1} \}
$$
is contained in one single class.
However, the set $\{w(R_\alpha,S_\alpha): w\in\Sigma_{< k+1} \}$
is in a one-to-one correspondence with the set
$\{w(R,S): w\in\Sigma_{< k+1} \}$ and hence for all $w_1,w_2\in\Sigma_{< k+1}$ with
$w_1(*_1,*_2)\neq w_2(*_1,*_2)$ we have $w_1(R_\alpha,S_\alpha)\neq w_2(R_\alpha,S_\alpha)$. In particular, this means that the nilprogression
$\{a w(R_\alpha,S_\alpha): w\in\Sigma_{< k+1} \}$ is non-degenerated, which
finishes the proof.
\end{proof}

\section{Proof of \cref{theorem:3}}
\label{section:nilpotent-PHJ}

We begin this section by discussing various properties of
polynomial mappings as defined in \cref{section:intro}.
It is shown in \cite{BL03} that for nilpotent groups $\G$ given polynomial mappings $P,Q:\FS\to\G$ both the reciprocal $P^{-1}$ and product $PQ$ are polynomial mappings.

We are particularly interested in collections of $\F$-measurable polynomial mappings
for idempotent filters $\F$ on a nilpotent group $\G$.
Let us give an example of such a setup.

\begin{Example}
\label{example:F-measurable-polynomials-on-Heisenberg}
Let $(x_n)$ and $(y_n)$ be sequences of positive integers. Let
$\G$ denote the discrete Heisenberg group, i.e.,
$$
\G=\left\{\begin{pmatrix}
1&a&c\\
&1&b\\
&&1
\end{pmatrix}:a,b,c\in\Z\right\},
$$
and recall that $\G$ is $2$-step nilpotent. Let
\[
V_N:=\left\{
\begin{pmatrix}
1&\sum\limits_{N\leq i\leq M}a_i x_i&\sum\limits_{N\leq\max\{i,j\}\leq M}c_{ij} x_iy_j\\
&1&\sum\limits_{N\leq j\leq M} b_j y_j\\
&&1
\end{pmatrix}:
\begin{array}{c}
M\geq N,\\
a_i,b_j\in \{0,1\},\\
c_{ij}\in \{-1,0,1\}
\end{array}
\right\}.
\]
Then the filter $\F=\{V\subset \G:\exists N\in\N~\text{s.t.}~V_N\subset V\}$
is an idempotent filter on $\G$, as can be checked by
straightforward calculations. For $\alpha,\beta\in\FS$ define
\begin{equation}
\label{eqn:x_alpha-and-y_alpha}
x_\alpha:=
\begin{pmatrix}
1&\sum\limits_{i\in\alpha}x_i&0\\
&1&0\\
&&1
\end{pmatrix},\qquad
y_\alpha:=
\begin{pmatrix}
1&0&0\\
&1&\sum\limits_{j\in\alpha}y_j\\
&&1
\end{pmatrix}
\end{equation}
and
\begin{equation}
\label{eqn:x_alpha-beta}
z_{\alpha\times\beta}:=
\begin{pmatrix}
1&0&\sum\limits_{(i,j)\in\alpha\times\beta}x_iy_j\\
&1&0\\
&&1
\end{pmatrix}.
\end{equation}
On the one hand, for any fixed $\beta\in\FS$
the maps $\alpha\mapsto x_\alpha$, $\alpha\mapsto y_\alpha$,
$\alpha\mapsto z_{\alpha\times\beta}$ and $\alpha\mapsto z_{\beta\times\alpha}$
are $\F$-measurable polynomial mappings of degree $1$.
On the other hand, the maps $\alpha\mapsto x_\alpha y_\alpha$,
$\alpha\mapsto  y_\alpha x_\alpha$
and $\alpha\mapsto z_{\alpha\times\alpha}$ are
$\F$-measurable polynomial mappings of degree $2$.
\end{Example}
 
In the following, we refer to any
finite collection of polynomial mappings
$\A=\{P_1,\ldots,P_m\}$ as a \Define{system}.

\begin{Proposition}[\cite{BL03}]
\label{proposition:weight}
Suppose $(\G,\cdot)$ is a nilpotent group.
There exists a set $\W$, called the \Define{set of weights},
endowed with a linear ordering $<$, such that for every system $\A$ there
is an element $w(\A)\in\W$ associated to $\A$,
referred to as the \Define{weight} of $\A$,
having the following
properties:
\begin{enumerate}	
[label=(\roman{enumi}),ref=(\roman{enumi}),leftmargin=*]
\item
Let $\Dshift{\beta}P$ denote the map
$$\Dshift{\beta}P(\alpha):=P^{-1}(\beta)P(\alpha\cup\beta).$$
For every system $\A$ there exists $\beta_0\in\FS$ such that for all $\beta\in\FS$ with $\beta>\beta_0$ one has $w(\A)=w(\A\cup \Dshift{\beta}\A)$.
\item
\label{weight7}
Let $P_{\low}\in\A$ and let $\A'=\{PP_{\low}^{-1}: P\in\A\}$.
If $w(\{P_{\low}\})\leq w(\{P\})$ for all
$P\in\A$ and $P_{\low}$ does not equal $1_\G$ $\FF$-a.e.~then $w(\A')<w(\A)$;
\item
The set $\W$ contains a minimal element $e_{\W}$ and
$w(\A)=e_{\W}$ if and only if $P= 1_{\G}$ $\FF$-a.e.~for all $P\in\A$;
\item
\label{weight9}
If $c_1,\ldots,c_m \in\G$, $\A$ is a system and
$\A'=\bigcup_{i=1}^m \A^{c_i}=\bigcup_{i=1}^m c_i^{-1}\A c_i$ then $w(\A)=w(\A')$.
\end{enumerate}
\end{Proposition}

\begin{Definition}[$\P$-minimal systems]
Let $\P$ be a good collection of $\F$-measurable polynomial mappings
for an idempotent filter $\F$ on a nilpotent group
$(\G,\cdot)$. We say that a system $\A$ is a
\Define{$\P$-minimal system} with \Define{$\P$-minimal element}
$P_{\min}$ if $P_{\min}$ is contained in $\A$ and
$\A P_{\min}^{-1}=\{PP_{\min}^{-1}:P\in\A\}$ is a subset of $\P$.
\end{Definition}

\begin{Lemma}
\label{lemma:f-minimal-conjugation}
Let $\P$ be a good collection of $\F$-measurable
polynomial mappings for an idempotent filter $\F$ on a
nilpotent group $(\G,\cdot)$. 
Suppose $\A$ is a $\P$-minimal systems with
$\P$-minimal element $P_{\min}$.
Then there exists $C\in\F$ such that for
all finite non-empty subsets $G\subset C$
the system $$\bigcup_{c\in G} \A^c =\{P^c: P\in\A,~c\in G\}$$
is $\P$-minimal with $\P$-minimal
element $P_{\min}$.
\end{Lemma}

\begin{proof}
First, note that every element $P\in\A$ can be written as
$RP_{\min}$ for some $R\in\P$. Then,
using part \ref{item:poly2} of \cref{def:gcofmp}, for every such $R$
we can find a set $C_R\in\F$ such that
$R^c [c,P_{\min}^{-1}]\in\P$ for all $c\in C_R$.
Let $C$ denote the intersection $\bigcap_RC_R$. Note that the set
$C$ belongs to $\F$, as it is an intersection of finitely many
sets contained in $\F$.

To finish the proof it suffices to show that
for every $P\in\A$ and for every $c\in C$ the
polynomial mapping $P^cP_{\min}^{-1}$ belongs to $\P$.
However,  simple algebra manipulations
show that
$P^cP_{\min}^{-1}=R^c [c,P_{\min}^{-1}]$, which finishes the proof.
\end{proof}

\begin{Lemma}
\label{lemma:f-minimal-derivative}
Let $\P$ be a good collection of $\F$-measurable
polynomial mappings for an idempotent filter $\F$ on a
nilpotent group $(\G,\cdot)$. 
Suppose $\A$ is a $\P$-minimal systems with
$\P$-minimal element $P_{\min}$ containing the constant polynomial
$1_{\G}$.
Then there exists $C\in\F$ and $\beta_0\in\FS$ with the following property: For all $u\in C$ and $\beta>\beta_0$ such that $P(\beta)u\in C$ for all $P\in \A$, we have that
$$
\A':=\A^u\cup \Dshift{\beta}\A^u
$$
is $\P$-minimal with $\P$-minimal
element $P_{\min}^{\new}(\alpha)$, which we can take to be
$\Dshift{\beta}P_{\min}^u$.
\end{Lemma}

\begin{proof}
First, note that every element $P\in\A$ can be written as
$RP_{\min}$ for some $R\in\P$. Define $c:=P_{\min}(\beta)u$.
On the one hand we have
\begin{eqnarray*}
\Dshift{\beta}P^u (P_{\min}^{\new})^{-1}
&=&c^{-1}P_{\min}(\beta)P^{-1}(\beta)
P(\alpha\cup\beta)P_{\min}^{-1}(\beta\cup\alpha)c\\
&=&c^{-1}\Dshift{\beta}(PP_{\min}^{-1})c\\
&=&\Dshift{\beta}R^c\\
&=&(RD_\beta R)^c\\
\end{eqnarray*}
On the other hand, we have
\begin{eqnarray*}
P^u (P_{\min}^{\new})^{-1}&=&
\big(u^{-1}P(\alpha)u\big)
\big(u^{-1}P_{\min}^{-1}(\alpha\cup\beta)P_{\min}(\beta)u\big)\\
&=&
u^{-1}R(\alpha)P_{\min}(\alpha)P_{\min}^{-1}(\beta\cup\alpha)P_{\min}(\beta)u\\
&=& (R D_\beta P_{\min}^{-1})^u
\end{eqnarray*}
Now the claim follows directly from part \ref{item:poly3} of \cref{def:gcofmp}.
\end{proof}

For the proof of \cref{theorem:3} we will use
PET-induction, a technique that was developed in \cite{BL96}, and which
proceeds by induction on the weight of systems as
defined in \cref{proposition:weight}.
However, this inductive process requires us to replace
\cref{theorem:3} with the stronger
\cref{theorem:VdW-fsyndetic-2} below, so that at every inductive step
we are able to rely on a strong enough induction hypothesis.

\begin{Maintheorem}\label{theorem:VdW-fsyndetic-2}
Let $\P$ be a good collection of $\F$-measurable polynomial mappings
for an idempotent filter $\F$ on a nilpotent group
$(\G,\cdot)$.
Let $\A$ be a $\P$-minimal system with $\P$-minimal element
$P_{\min}$, partitioned into three classes,
$\A=\A^{-}\cup \{1_\G\}\cup \A^{+}$. Then
for all $\beta \in \FS$, for all $V\in\F$ and for all piecewise
$\F$-syndetic sets $A$ there
exist $\alpha\in \FS$ with
$\alpha>\beta$ and $v\in V$ such that 
\begin{align*} 
P(\alpha)v\in V,\qquad\qquad&\forall~P\in \A^-,
\\
P(\alpha)v\in A,\qquad\qquad&\forall~P\in \{1_\G\}\cup \A^+.
\end{align*}
\end{Maintheorem}

\cref{theorem:3} is indeed an immediate consequence of
\cref{theorem:VdW-fsyndetic-2}, because one can choose $\A^{-}$
to be the empty set.

The following lemma will be instrumental in
proving \cref{theorem:VdW-fsyndetic-2}.

\begin{Lemma}[Color Focusing]\label{lemma:color-focusing}
Let $\P$ be a good collection of $\F$-measurable polynomial mappings
for an idempotent filter $\F$ on a nilpotent group
$(\G,\cdot)$.
Suppose $\A=\A^{-}\cup \{1_\G\}\cup \A^{+}$ is a given $\P$-minimal
system with $\P$-minimal element $P_{\min}$ and assume that
\cref{theorem:VdW-fsyndetic-2} has already
been proven for all $\P$-minimal systems
$\mathscr{B}=\mathscr{B}^-\cup\{1_{\G}\}\cup \mathscr{B}^+$ with
$w(\mathscr{B}^+)<w(\A^+)$. Let  $G$ be a finite non-empty subset of
$\G$ such that $\A_0=\bigcup_{c\in G} \A^c$ is also a
$\P$-minimal system with the same $\P$-minimal element $P_{\min}$.
Assume $B$ is a subset of $\G$ such that
$G^{-1}B\in\F$. Then, for every $s\geq 1$, one of the following two
cases holds:
\begin{enumerate}	
[label={\normalfont(\arabic{enumi})}~~,ref={\normalfont(\arabic{enumi})},
leftmargin=*]
\item
\label{item:color-focusing-1}
For all $U\in\F$ and $\alpha_0\in\FS$ there exist $u_s\in U$, 
$\alpha_1,\ldots,\alpha_s\in\FS$ with
$\alpha_0<\alpha_1<\ldots<\alpha_s$ and $s$ distinct elements
$c_1, \ldots, c_s \in G$ such that
\begin{align}
\label{equation:counterweight}
&P(\alpha_j\cup\ldots\cup\alpha_s)u_s\in U,
~~~~~~~~~~~~\forall~P\in \A_0^-\cup\{1_{\G}\},
\\
\label{equation:weight}
&P(\alpha_j\cup\ldots\cup\alpha_s)u_s\in c_j^{-1}B,
~~~~~~\forall~P\in \A_0^+.
\end{align}
Moreover the system
$\A_s=\A_0^{u_s}\cup\Dshift{\alpha_s}\A_0^{u_s}\cup\ldots\cup 
\Dshift{\alpha_1\cup\ldots\cup\alpha_s}\A_0^{u_s}$
remains $\P$-minimal with $\P$-minimal element
$\Dshift{\alpha_1\cup\ldots\cup\alpha_s}P_{\min}^{u_s}$.
\item
\label{item:color-focusing-2}
For all $U\in\F$ and $\alpha_0\in\FS$ there exist $\alpha \in \FS$ with
$\alpha>\alpha_0$, $c\in G$ and $v\in cU$, such that 
\begin{align*}
P(\alpha)v\in&~cU,~~~~~~~~~~\forall~P\in \A^-,
\\
P(\alpha)v\in&~B,~~~~~~~~~~~~\forall~P\in \{1_{\G}\}\cup \A^+.
\end{align*}
\end{enumerate}
\end{Lemma}

\begin{Remark}
In many classical proofs of \ref{theorem:vdW} and the Polynomial van der Waerden theorem (cf. \cite{GRS90,Walters00}), monochromatic configurations of the form $\{P(\alpha)u:P\in \A\}$ are referred to as \emph{sets focused at $u$}. Moreover, a finite collection of sets focused at the same point $u$ with the property that no two sets have the same ``color'' (as it is the case in equation \eqref{equation:weight}) are referred to as \emph{a collection of color focused sets}. The inductive procedure of constructing larger and larger families of color focused sets is then called the \emph{color focusing argument}. \cref{lemma:color-focusing} can be thought of as a filter-sensitive generalization of the original color focusing argument.
\end{Remark}

\begin{proof}[Proof of \cref{lemma:color-focusing}]
In this proof it will be convenient to identify polynomial mappings
with their equivalence class of $\FF$-a.e.~equivalent
polynomial mappings. This is allowed since the statement of
\cref{lemma:color-focusing} as well as all proceeding arguments
in this proof are insensitive to replacing
polynomial mappings with elements in their $\FF$-a.e.~equivalence class.

We proceed by induction on $s$ and start with $s=1$.
We can write $\A_0$ as $\A_0^{-}\cup \{1_\G\}\cup \A_0^{+}$,
where $\A_0^{\pm}=\bigcup_{c\in G} c^{-1}\A^{\pm} c$.
We know that $w(\A^+)=w(\A_0^+)$, because of
\cref{proposition:weight}, part \ref{weight9}.
Since we only care about $\FF$-a.e.~equivalency classes,
we tacitly assume that in the decomposition
$\A_0=\A_0^{-}\cup \{1_\G\}\cup \A_0^{+}$ all polynomial mappings
in $\A_0$ that are $\FF$-a.e.~equal to $1_\G$ are grouped with
$\{1_\G\}$ and that $\A_0^{+}$ and $\A_0^{-}$ contain no
polynomial mapping that is $\FF$-a.e.~equal to $1_\G$.
Let us pick $P_{\low}\in \A_0^+$ such that
$w(\{P_{\low}\})\leq w(\{P\})$ for all $P\in\A_0^+$.
Define 
\[
\mathscr{B}_0^- =
\big\{PP_{\low}^{-1}: P\in \A_0^-\cup\{1_{\G}\}\big\}~~\text{and}~~
\mathscr{B}_0^+=\big\{PP_{\low}^{-1}:
P\in\A_0^+\backslash\{P_{\low}\}\big\}.
\]
It follows from
\cref{proposition:weight}, part \ref{weight7},
that $w(\mathscr{B}_0^+)<w(\A_0^+)$.
Also, after a moment's consideration, we see that 
$\mathscr{B}_0=\mathscr{B}_0^-\cup\{1_{\G}\}\cup \mathscr{B}_0^+$ is
an $\P$-minimal system with
$\P$-minimal element $Q_{\min}=P_{\min}^{~}P_{\low}^{-1}$.
Using \cref{lemma:f-minimal-derivative}
we can find an $U_0\in\F$ with $U_0\subset U$
and $\beta_0\in\FS$ with $\beta_0>\alpha_0$ such that
$\A_0^{u_1}\cup \Dshift{\alpha_1}\A^{u_1}$ is
$\P$-minimal with $\P$-minimal element
$\Dshift{\alpha_1}P_{\min}^{u_1}$ whenever $\alpha_{1}>\beta_0$
and $P(\alpha_1)u_1 \in U_0$ for all
$P\in\A_0$.

Since $G^{-1}B\in\F$, there exists some
$c_1\in G$ such that $c_1^{-1}B$ is piecewise $\F$-syndetic (cf.
\cref{rem:partition-regularity-pw-f-syndetic}).
Let us put $N_0:= c_1^{-1}B\cap U_0$ and let us note that
the set $N_0$ is piecewise $\F$-syndetic, as it is
an intersection of a piecewise $\F$-syndetic set and a set belonging to $\F$.  
Since the weight of the system $\mathscr{B}_0^+$ is strictly
smaller than the weight of $\A^+$, we can apply
\cref{theorem:VdW-fsyndetic-2} in order
to find $\alpha_1>\beta_0$ and $u_1'\in U_0$ such that
\begin{align*}
Q(\alpha_1)u_1'\in&~U_0,
~~~~~~~\forall~Q\in \mathscr{B}_0^-,
\\
Q(\alpha_1)u_1'\in&~N_0,
~~~~~\forall~Q\in\{1_\G\}\cup\mathscr{B}_0^+.
\end{align*}
If we put
$u_1= P_{\low}^{-1}(\alpha_1)u_1'$,
then a simple calculation shows that $Q(\alpha_1)u_1'=P(\alpha_1)u_1$.
With this choice of $u_1$, $\alpha_1$ and $c_1$
equations \eqref{equation:counterweight}
and \eqref{equation:weight} are satisfied and the system
$\A_0^{u_1}\cup \Dshift{\alpha_1}\A^{u_1}$ remains $\P$-minimal.
This concludes the case $s=1$.

Next, let us deal with the inductive step, $s\to s+1$. 
Take any ultrafilter $q\in K(\FilterClosure{\F})$ with $B\in q$
(cf. \cref{theorem:SZZ-2.3}).
Let $B':= B/q$ and let $U':=(U/\F)\cap((G^{-1}B)/\F)$.
Observe that $B'$ is $\F$-syndetic, by virtue of
\cref{theorem:SZZ-2.2}, and that $U'\in\F$. Also, since $G^{-1}B\in\F$, it follows that
$(G^{-1}B)/q \in \F$. A simple calculation then shows that
$G^{-1}(B/q)=(G^{-1}B)/q$, which tells us that $G^{-1}B'\in\F$.

This means we can apply the induction hypothesis to $U'$ and $B'$
in order to find
$u_s\in U'$, elements $\alpha_1,\ldots,\alpha_s\in \FS$ with
$\alpha_0<\alpha_1<\ldots<\alpha_s$, and
$s$ distinct `colors' $c_1,\ldots,c_s\in G$ such that
\begin{align*}
&P(\alpha_j\cup\ldots\cup\alpha_r)u_s\in U',~~~~~~~~~~~~~~
\forall~P\in \A_0^-\cup\{1_\G\},
\\
&P(\alpha_j\cup\ldots\cup\alpha_r)u_s\in c_j^{-1}B',~~~~~~~~
\forall~P\in \A_0^+.
\end{align*}
This implies that there exist sets $B''\in q$ and $U''\in \F$
such that
\begin{align}
\label{equation:counterweight-2}
&P(\alpha_j\cup\ldots\cup\alpha_s)u_sU''\subset U,
~~~~~~~~~~~~~~\forall~P\in \A_0^-\cup\{1_\G\},
\\
\label{equation:weight-2}
&P(\alpha_j\cup\ldots\cup\alpha_s)u_sB''\subset c_j^{-1}B,
~~~~~~~~\forall~P\in \A_0^+.
\end{align}
Since $u_s\in U/\F$ we may assume $u_sU''\subset U$, because otherwise
we can replace $U''$ with $U''\cap u_s^{-1}U$.
Analogously, since $u_s\in (G^{-1}B)/\F$ we may assume that
$u_sB''\subset G^{-1}B$ because otherwise we can replace $B''$
with $B''\cap u_s^{-1} G^{-1}B$.
Since $G^{-1}B$ covers $u_sB''$, there exists a piecewise
$\F$-syndetic subset
$N_s\subset B''$ and a `color' $c_{s+1} \in F$ such that
$u_sN_s\subset c_{s+1}^{-1}B$. 

Now we have to distinguish two cases. The first case is $c_{r+1}=c_j$ for some $j\in \{1,\ldots,r\}$. In this case one may
take any $n\in U''\cap N_s$ and put $c:=c_{r+1}$,
$v:=cu_sn$ and
$\alpha:=\alpha_j\cup\ldots\cup \alpha_r$.
With this choice of $\alpha\in\FS$, $c\in G$ and $v\in cU$ we are
in case \ref{item:color-focusing-2} of
\cref{lemma:color-focusing} and therefore we are done with the
current inductive step $s\to s+1$, as well as with all subsequent
inductive steps, and the inductive process terminates here.

The second case is when $c_{s+1}\neq c_j$ for all $j\in \{1,\ldots,s\}$.
If this is the case then we proceed as follows. Define 
$$
\A_s^{\pm}:=u_s^{-1}\big(\A_0^{\pm}\cup \Dshift{\alpha_s}\A_0^{\pm}\cup
\ldots\cup \Dshift{\alpha_1\cup\ldots\cup\alpha_s}\A_0^{\pm}\big)u_s.
$$
Under the assumptions
of the induction hypothesis, the system $\A_s$ is $\P$-minimal with
$\P$-minimal element $\Dshift{\alpha_1\cup\ldots\cup\alpha_s}P_{\min}^{u_s}$. 
Using \cref{lemma:f-minimal-derivative} we can find $U_s\in \F$ with $U_s\subset U''$
and $\beta_s\in\FS$ with $\beta_s>\alpha_s$ such that
$\A_s^{u}\cup \Dshift{\alpha_{s+1}}\A_s^u$ is
$\P$-minimal with $\P$-minimal element
$\Dshift{\alpha_1\cup\ldots\cup\alpha_{s+1}}P_{\min}^{u_su}$
whenever $\alpha_{s+1}>\beta_s$ and $P(\alpha_{s+1})u\in U_s$ for all
$P\in\A_s$. Let us put $N_s:= B''\cap U_s$. Since $B''$ is contained in
$q$ and $U_s$ is contained in $\F$, it follows that $N_s$
is contained in $q$. Hence $N_s$ is piecewise $\F$-syndetic.

Let $P_{\low}$ denote the element of
$\A_s^+$ of lowest weight, i.e. $w(\{P_{\low}\})\leq w(\{P\})$ for all
$P\in \A_s^+$.
Again, we assume without loss of generality that
$\A_s^+$ contains no polynomial mappings that are
$\FF$-a.e.~equivalent to $1_\G$. In particular,
$P_{\low}$ is not $\FF$-a.e.~equal to $1_\G$.
Define 
\[
\mathscr{B}_s^- =
\big\{PP_{\low}^{-1}: P\in \A_s^-\cup\{1_{\G}\}\big\}~~\text{and}~~
\mathscr{B}_s^+=\big\{PP_{\low}^{-1}:
P\in\A_s^+\backslash\{P_{\low}\}\big\}.
\]
We have $w(\mathscr{B}_s^+)<w(\A^+)$ by \cref{proposition:weight}, part \ref{weight7}.
Also, $\mathscr{B}_s=\mathscr{B}_s^-\cup\{1_{\G}\}\cup \mathscr{B}_s^+$ is
an $\P$-minimal system with
$\P$-minimal element
$Q_{\min}=\Dshift{\alpha_1\cup\ldots\cup\alpha_s}P_{\min}^u P_{\low}^{-1}$.
We can now apply \cref{theorem:VdW-fsyndetic-2} to find $\alpha_{s+1}>\beta_s$ and $u_{s+1}'\in U_s$ such that
\begin{align}
\label{equation:counterweight-3}
Q(\alpha_{s+1})u_{s+1}'\in&~U_s,
~~~~~~~\forall~Q\in \mathscr{B}_s^-,\\
\label{equation:weight-3}
Q(\alpha_{s+1})u_{s+1}'\in&~N_s,
~~~~~\forall~Q\in\{1_\G\}\cup\mathscr{B}_s^+.
\end{align} 
Finally, define $u_{s+1}:=u_sP_{\low}^{-1}(\alpha_{s+1})u_{s+1}'$.
For every $P\in\A_0$ and every $j\in[s]$ there exists $Q\in\B_s$
such that
\begin{eqnarray*}
Q(\alpha_{s+1})u_{s+1}'&=&
u_s^{-1}P^{-1}(\alpha_j\cup\ldots\cup\alpha_s)P(\alpha_j\cup\ldots\cup\alpha_{s+1})u_s
P_{\low}^{-1}(\alpha_{s+1})u_{s+1}'\\
&=&u_s^{-1}P^{-1}(\alpha_j\cup\ldots\cup\alpha_s)P(\alpha_j\cup\ldots\cup\alpha_{s+1})u_{s+1}.
\end{eqnarray*}
If we combine equations
\eqref{equation:counterweight-3} and \eqref{equation:weight-3} with
equations \eqref{equation:counterweight-2} and \eqref{equation:weight-2},
we obtain
\begin{align*}
&P(\alpha_j\cup\ldots\cup\alpha_{s+1})u_{s+1}\subset U,
~~~~~~~~~~~~~~\forall~P\in \A_0^-\cup\{1_\G\},
\\
&P(\alpha_j\cup\ldots\cup\alpha_{s+1})u_{s+1}\subset c_j^{-1}B,
~~~~~~~~\forall~P\in \A_0^+
\end{align*} 
for all $j\in\{1,\ldots,s\}$. For the case $j=s+1$ we simply note that
$u_sU_S\subset u_sU''\subset U$
and $u_sN_s\subset c_{s+1}^{-1}B$ and therefore it
follows from \eqref{equation:counterweight-3} and \eqref{equation:weight-3} and $\A_0^{u_s}\subset \A_s$ that
\begin{align*}
&P(\alpha_{s+1})u_{s+1}\subset U,
~~~~~~~~~~~~~~\forall~P\in \A_0^-\cup\{1_\G\},
\\
&P(\alpha_{s+1})u_{s+1}\subset c_{s+1}^{-1}B,
~~~~~~~~\forall~P\in \A_0^+.
\end{align*}
Also, the newly created system $\A_{s+1}$ is $\P$-minimal
because $\alpha_{s+1}>\beta_s$.
This completes the inductive step $s\to s+1$.
\end{proof}

\begin{proof}[Proof of \cref{theorem:VdW-fsyndetic-2}]
Since $A$ is piecewise $\F$-syndetic, there exists $p \in K(\FilterClosure{\F})$ such that $A \in p$.
Let $A'$ denote the $\F$-syndetic set $A/p$ (cf. \cref{theorem:SZZ-2.2}) and put $V':=V/p$.
Since $\F$ is idempotent we have $V' \in\F$ (since $V/\F \subset V/p$) and $V'/\F \in \F$.
Pick $C\in\F$ as guaranteed by \cref{lemma:f-minimal-conjugation}.
Since $A'$ is $\F$-syndetic, we can find a finite non-empty set $G\subset (V'/\F\cap C)$ such that $G^{-1}A' \in \F$.
Define $U:=\bigcap_{c\in G}c^{-1}V'$.
As an intersection of elements in $\F$, the set $U$ belongs to $\F$.

We now claim that for all $\beta\in\FS$ there exist $\alpha \in \FS$ with $\alpha>\beta$, $c\in G$ and $w\in \G$, such that 
\begin{align*}
P(\alpha)w\in&~cU,~~~~~~~~~~\forall~P\in \A^-,
\\
P(\alpha)w\in&~A',~~~~~~~~~~~~\forall~P\in \{1_{\G}\}\cup \A^+.
\end{align*}

Before we verify this claim, let us show how it can be used to finish the proof of \cref{theorem:VdW-fsyndetic-2}. Note that $cU\subset V'$ and recall that $V'=V/p$ and $A'=A/p$. Hence, there exists a set $N\in p$ such that
\begin{align*}
P(\alpha)w N\in&~ V,~~~~~~~~~~~~~\forall~P\in \A^-,
\\
P(\alpha)w N\in&~A,~~~~~~~~~~~\forall~P\in \{1_{\G}\}\cup \A^+.
\end{align*}
We can now choose $v$ to be any element in $wN$ and the proof is completed.

It remains to prove the above claim, which is done by induction on the weight of $\A^+$.
The beginning of the induction is given by the case $\A^+=\emptyset$.
Let $U':=U\cap G^{-1}A'$. Since $U'\in\F$ and $\F$ is idempotent, it follows that $U'/\F \in \F$.
Let $P_{\min}$ denote the
$\P$-minimal element of $\A=\A^-\cup\{1_\G\}$.
Note that $P^c P_{\min}^{-1}$ is $\F$-measurable for all $P\in \A^-\cup\{1_\G\}$ and for all $c \in G$.
Hence for all $\beta\in\FS$ there exist $\alpha \in \FS$ with $\alpha>\beta$ such that
$(P^cP_{\min}^{-1})(\alpha)\in U'/\F$ for all $P\in \A^-\cup\{1_\G\}$ and for all $c \in G$.

This implies that $P_{\min}(\alpha)G^{-1}A'\in \F$ as well as
$P_{\min}(\alpha)c^{-1}P^{-1}(\alpha)cU\in\F$ for all $P\in \A^-$ and for all $c \in G$.
In particular, the intersection of $P_{\min}(\alpha)G^{-1}A'$ with
$\bigcap_{c \in G}\bigcap_{P\in \A^-} P_{\min}(\alpha)c^{-1}P^{-1}(\alpha)cU$ is non-empty. Let $u$ be an arbitrary element in this intersection. Choose $c\in \G$ such that $u\in P_{\min}(\alpha)c^{-1}A'$ and set $w:=cP_{\min}^{-1}(\alpha)u$. Clearly, $w\in A'$ and $P(\alpha)w\in cU$ for all $P\in \A^-$. This completes the initial step of the induction.

For the proof of the inductive step assume that \cref{theorem:VdW-fsyndetic-2} has already been proven for all systems $\B^-\cup\{1_\G\}\cup \B^+$ with $w(\B^+)<w(\A^+)$.
We apply \cref{lemma:color-focusing} with $s=|G|+1$.
Since $s> |G|$ we cannot be in the case \ref{item:color-focusing-1} of \cref{lemma:color-focusing}; therefore we have to be in case \ref{item:color-focusing-2}.
This means we can find  $\alpha>\beta$, $c\in G$ and $w\in U$ such that 
\begin{align*}
P(\alpha)w\in&~cU,~~~~~~~~~~\forall~P\in \A^-,
\\
P(\alpha)w\in&~A',~~~~~~~~~\forall~P\in \{1_{\G}\}\cup \A^+,
\end{align*}
which completes the proof.
\end{proof}

\bibliographystyle{siam}
\providecommand{\noopsort}[1]{} 
\allowdisplaybreaks
\small
\bibliography{Bib.bib}

\bigskip
\footnotesize

\noindent
John H.~Johnson Jr.\\
\textsc{Department of Mathematics, Ohio State University, Columbus, OH 43210, USA}\par\nopagebreak
\noindent
\href{mailto:johnson.5316@osu.edu}
{\texttt{johnson.5316@osu.edu}}

\medskip

\noindent
Florian K.\ Richter\\
\textsc{Department of Mathematics, Ohio State University, Columbus, OH 43210, USA}\par\nopagebreak
\noindent
\href{mailto:richter.109@osu.edu}
{\texttt{richter.109@osu.edu}}
\end{document}